\documentclass[10pt, leqno]{amsart}
\usepackage{amsfonts,amssymb}

\usepackage{amsmath}
\usepackage{amsthm}
\usepackage{amsrefs}
\usepackage{qsymbols}
\usepackage{latexsym}
\usepackage{chngcntr}
\usepackage[noadjust]{cite}
\usepackage{paralist}
\usepackage{esint}

\newtheorem{theorem}{Theorem}[section]

\newtheorem{lemma}[theorem]{Lemma}
\newtheorem{proposition}[theorem]{Proposition}

\theoremstyle{definition}

\newcommand{\IR}{\mathbb{R}}
\newcommand{\IC}{\mathbb{C}}
\newcommand{\IN}{\mathbb{N}}

\newcommand{\IP}{\mathbb{P}}

\newcommand{\cM}{\mathcal{M}}

\newcommand{\cR}{\mathcal{R}}
\newcommand{\cA}{\mathcal{A}}
\newcommand{\cB}{\mathcal{B}}

\newcommand{\cT}{\mathcal{T}}

\renewcommand{\L}{\mathrm{L}}
\newcommand{\C}{\mathrm{C}}

\renewcommand{\H}{\mathrm{H}}

\renewcommand{\S}{\mathrm{S}}
\newcommand{\W}{\mathrm{W}}

\newcommand{\M}{\mathrm{M}}

\newcommand{\Lloc}{\L_{\mathrm{loc}}}

\newcommand{\fa}{\mathfrak{a}}
\newcommand{\fb}{\mathfrak{b}}

\newcommand{\abs}[1]{\left|#1\right|}

\newcommand{\e}{\mathrm{e}}

\renewcommand{\d}{\mathrm{d}}
\newcommand{\eps}{\varepsilon}
\newcommand{\loc}{\mathrm{loc}}
\renewcommand\Re{\operatorname{Re}}

\newcommand{\Lop}{\mathcal{L}}
\newcommand{\pv}{\mathrm{p.v.}}

\DeclareMathOperator{\supp}{supp}

\DeclareMathOperator{\diam}{diam}

\DeclareMathOperator{\Id}{Id}

\DeclareMathOperator{\dom}{\mathcal{D}}

\numberwithin{equation}{section}

\title[$\L^p$-extrapolation of non-local operators]{$\L^p$-extrapolation of non-local operators: \\ Maximal regularity of elliptic integrodifferential operators with measurable coefficients}
\author{Patrick Tolksdorf}
\address{Institut f\"ur Mathematik, Johannes Gutenberg-Universit\"at Mainz, Staudingerweg 9, 55099 Mainz, Germany}
\email{tolksdorf@uni-mainz.de}

\subjclass[2010]{}
\date{\today}
\thanks{}
\dedicatory{In the occasion of the $60$th birthday of Matthias Hieber}

\begin{document}
\begin{abstract}
The aim of this article is to deepen the understanding of the derivation of $\L^p$-estimates of non-local operators. We review the $\L^p$-extrapolation theorem of Shen~\cite{Shen} which builds on a real variable argument of Caffarelli and Peral~\cite{Caffarelli_Peral} and adapt this theorem to account for non-local weak reverse H\"older estimates. These non-local weak reverse H\"older estimates appear for example in the investigation of non-local elliptic integrodifferential operators. This originates from the fact that here only a non-local Caccioppoli inequality is valid, see Kuusi, Mingione, and Sire~\cite{Kuusi_Mingione_Sire}. As an application, we prove resolvent estimates and maximal regularity properties in $\L^p$-spaces of non-local elliptic integrodifferential operators.
\end{abstract}
\maketitle

\section{Introduction}

\noindent Non-local phenomena play a major role in many different areas in the study of partial differential equations~\cite{Craig_Schanz_Sulem, Toland, Alberti_Bellettini, Cabre_Sola-Morales, Sire_Valdinoci, Caffarelli_Roquejoffre_Sire, Caffarelli_Vasseur, Majda_Tabak}. In particular, in mathematical fluid mechanics non-local phenomena arise naturally due to the presence of the pressure and the imposed solenoidality of the velocity field. One prominent example of a non-local operator in the study of mathematical fluid mechanics is the Stokes operator that is given --- if the underlying domain is regular enough --- by the Helmholtz projection $\IP$ applied to the Laplacian $- \Delta$. \par
Another prominent example is the so-called dissipative surface quasi-geostrophic equation. It is often studied as a model equation to understand non-linear mechanisms~\cite{Constantin, Majda_Tabak} connected to the Navier--Stokes equations. The dissipative surface quasi-geostrophic equation involves the fractional Laplacian which is given in the whole space for $\alpha \in (0 , 1)$ and $u \in \C_c^{\infty} (\IR^d)$ by
\begin{align}
\label{Eq: Fractional Laplacian}
 [(- \Delta)^{\alpha} u] (x) := C_{d , \alpha} \, \pv \int_{\IR^d} \frac{u(x) - u(y)}{\lvert x - y \rvert^{d + 2 \alpha}} \; \d y.
\end{align}
Here, $C_{d , \alpha} > 0$ denotes a suitable normalization constant. \par
For both types of operators certain mapping properties cease to exist in irregular situations. Indeed, it is well-known that the Stokes operator does \textit{not} generate a strongly continuous semigroup on $\L^p_{\sigma}$ if $p > 2$ is large enough and if it is considered in a bounded Lipschitz domain~\cite{Deuring}. Also the Riesz-transform $\nabla (- \Delta_D)^{- 1 / 2}$ is not bounded in general on $\L^p$ for $p > 3$ if it is considered in an arbitrary bounded Lipschitz domain~\cite{Jerison_Kenig}. Here, $- \Delta_D$ denotes the Dirichlet Laplacian on the underlying domain. Such phenomena do not only occur in the presence of an irregular boundary, but also in smooth geometric constellations in the presence of irregular coefficients. For example, if $A = - \nabla \cdot \mu \nabla$ denotes an elliptic operator with complex $\L^{\infty}$-coefficients, an example by Frehse~\cite{Frehse, Hofmann_Mayboroda_McIntosh} shows that for any $p > 2d / (d - 2)$ there exists a complex-valued, strongly elliptic matrix $\mu \in \L^{\infty}$ such that $(\Id + A)^{-1}$ does not map $\L^p$ into itself. Here, the dimension $d$ satisfies $d \geq 3$. \par
Important tools to reveal for which numbers $p \in (1 , \infty)$ certain $\L^p$-mapping properties do still hold are so-called \textit{$p$-sensitive Calder\'on--Zygmund theorems}. Amongst others, there is the $\L^p$-extrapolation theorem of Shen~\cite{Shen} which builds on a real variable argument of Caffarelli and Peral~\cite[Sec.~1]{Caffarelli_Peral}. This argument was already successfully applied to reveal properties of the Stokes operator in Lipschitz domains, Riesz-transforms of elliptic operators, homogenization theory of elliptic operators, maximal regularity properties of elliptic operators, and elliptic boundary value problems~\cite{Caffarelli_Peral, Shen, Shen_2, Shen_3, Shen_4, Wei_Zhang, Kenig_Lin_Shen, Tolksdorf, terElst_et_al}. In its whole space version, this $\L^p$-extrapolation theorem reads as follows, cf.~\cite[Thm.~3.1]{Shen}. For its formulation we denote by $\Lop(E , F)$ set of all bounded linear operators between two Banach spaces $E$ and $F$ and by $\fint_B \d x$ the mean value integral over a measurable set $B$ with $0 < \lvert B \rvert < \infty$.

\begin{theorem}[Shen]
\label{Thm: Shen}
Let $T \in \Lop(\L^2 (\IR^d) , \L^2 (\IR^d))$ be a bounded linear operator. Assume there exist $p > 2$, $\iota_2 > \iota_1 > 1$, and $C > 0$ such that for all $x_0 \in \IR^d$, $r > 0$, and all compactly supported functions $f \in \L^{\infty} (\IR^d)$ with $f \equiv 0$ in $B(x_0 , \iota_2 r)$ the inequality
\begin{align}
\label{Eq: Original Shen}
 \bigg( \fint_{B(x_0 , r)} \lvert T f \rvert^p \; \d x \bigg)^{\frac{1}{p}} \leq C \bigg\{ \bigg( \fint_{B(x_0 , \iota_1 r)} \lvert T f \rvert^2 \; \d x \bigg)^{\frac{1}{2}} + \sup_{B^{\prime} \supset B} \bigg( \fint_{B^{\prime}} \lvert f \rvert^2 \; \d x \bigg)^{\frac{1}{2}} \bigg\}
\end{align}
holds. Here, the supremum runs over all balls $B^{\prime}$ containing $B$. \par
Then, for all $2 < q < p$ the restriction of the operator $T$ to $\L^2 (\IR^d) \cap \L^q (\IR^d)$ extends to a bounded operator on $\L^q (\IR^d)$. Furthermore, the $\L^q$-operator norm of $T$ can be quantified by the constants above.
\end{theorem}

See~\cite[Thm.~3.3]{Shen} for a version of this theorem in bounded Lipschitz domains and~\cite[Thm.~4.1]{Tolksdorf} for an extension to Lebesgue measurable sets and the vector-valued case. If an estimate of the form~\eqref{Eq: Original Shen} can be established but without the term involving the supremum, the corresponding inequality is called a weak reverse H\"older estimate. \par
In~\cite{Tolksdorf} and~\cite{terElst_et_al} one finds applications of Shen's $\L^p$-extrapolation theorem to establish mapping properties of resolvents of elliptic operators in divergence form with irregular coefficients and in irregular domains. The ingredients to establish the required weak reverse H\"older estimates in the corresponding situations are well-known as only Sobolev's embedding, Caccioppoli's inequality, and Moser's iteration are used. However, these basic ingredients suffice to unveil optimal mapping properties for the resolvent operators. Unfortunately, these basic techniques cease to work to establish a weak reverse H\"older estimate suitable for Theorem~\ref{Thm: Shen} if the elliptic operator in divergence form is replaced by an operator that is non-local as for example the Stokes operator. Motivated by this fact, we study here as ``toy-operators'' non-local elliptic integrodifferential operators of order $2 \alpha$, where $\alpha \in (0 , 1)$. These operators generalize the fractional Laplacian in the whole space given by~\eqref{Eq: Fractional Laplacian} and are defined via the form method as follows: \par
Define the for $\alpha \in (0 , 1)$ the fractional Sobolev space $\W^{\alpha , 2} (\IR^d)$ by
\begin{align*}
 \W^{\alpha , 2} (\IR^d) := \{ f \in \L^2 (\IR^d) : \| f \|_{\W^{\alpha , 2} (\IR^d)} < \infty \}
\end{align*}
where
\begin{align*}
 \| f \|_{\W^{\alpha , 2} (\IR^d)}^2 := \| f \|_{\L^2 (\IR^d)}^2 + \int_{\IR^d} \int_{\IR^d} \frac{\lvert f (x) - f (y) \rvert^2}{\lvert x - y \rvert^{d + 2 \alpha}} \; \d x \; \d y.
\end{align*}
Having second-order elliptic operators in divergence form in mind, we generalize~\eqref{Eq: Fractional Laplacian} the variational definition of by considering the sesquilinear form defined by
\begin{align}
\label{Eq: Sesquilinear form}
\begin{aligned}
 \fa : \W^{\alpha , 2} (\IR^d) \times \W^{\alpha , 2} (\IR^d) &\to \IC \\
 (u , v) &\mapsto \int_{\IR^d} \int_{\IR^d} K (x , y) (u (x) - u (v)) \overline{(v(x) - v(y))} \; \d x \; \d y,
\end{aligned}
\end{align}
where $K : \IR^d \times \IR^d \to \IC$ is measurable and satisfies for some $0 < \Lambda < 1$ the ellipticity and boundedness conditions
\begin{align}
\label{Eq: Ellipticity}
 0 < \frac{\Lambda}{\lvert x - y \rvert^{d + 2 \alpha}} \leq \Re(K (x , y)) \leq \lvert K(x , y) \rvert \leq \frac{\Lambda^{- 1}}{\lvert x - y \rvert^{d + 2 \alpha}} \qquad (\text{a.e.\@ } x , y \in \IR^d).
\end{align}
Define the realization $A$ of $\fa$ on $\L^2 (\IR^d)$ by
\begin{align*}
 \dom(A) := \{ u \in \W^{\alpha , 2} (\IR^d) : \exists f \in \L^2 (\IR^d) \text{ such that } \fa (u , v) = \langle f , v \rangle_{\L^2} \; \forall v \in \W^{\alpha , 2} (\IR^d) \}
\end{align*}
and for $u \in \dom(A)$ with associated function $f$ the image of $A$ under $u$ is defined as
\begin{align*}
 A u := f.
\end{align*}

\indent Recently, there was a brisk interest in such operators and in certain non-linear counterparts~\cite{Kassmann, Kuusi_Mingione_Sire, Bass_Ren, Biccari_Warma_Zuazua, Biccari_Warma_Zuazua_2, Leonori_Peral_Primo_Soria, Schikorra, Auscher_Bortz_Egert_Saari}. We would like to highlight the work of Kuusi, Mingione, and Sire~\cite{Kuusi_Mingione_Sire} where the following non-local Caccioppoli-type inequality was proven
\begin{align*}
 &\int_{B(x_0 , 2 r)} \int_{B(x_0 , 2 r)} \frac{\lvert u (x) \eta (x) - u (y) \eta (y) \rvert^2}{\lvert x - y \rvert^{d + 2 \alpha}} \; \d x \; \d y \\
 &\qquad \leq C \bigg( \frac{1}{r^{2 \alpha}} \int_{B(x_0 , 2 r)} \lvert u (x) \rvert^2 \; \d x + \int_{B(x_0 , 2 r)} \lvert u (x) \rvert \; \d x \int_{\IR^d \setminus B(x_0 , 2 r)} \frac{\lvert u (y) \rvert}{\lvert x_0 - y \rvert^{d + 2 \alpha}} \; \d y \bigg)
\end{align*}
for functions $u \in \W^{\alpha , 2} (\IR^d)$ that satisfy
\begin{align}
\label{Eq: Variational zero}
 \fa (u , v) = 0 \qquad (v \in \C_c^{\infty} (B(x_0 , 3 r)))
\end{align}
and for any smooth function $\eta$ with $0 \leq \eta \leq 1$, $\eta = 1$ in $B(x_0 , r)$, $\eta \equiv 0$ in $\IR^d \setminus B(x_0 , 2 r)$, and $\| \nabla \eta \|_{\L^{\infty}} \leq C / r$. The authors in~\cite{Kuusi_Mingione_Sire} proved this inequality under slightly different assumptions on the sesquilinear form, the most important is that they considered even non-linear equations. In contrast to the results in~\cite{Kuusi_Mingione_Sire}, we consider here also complex-valued coefficients $K$ what can be seen as a minor improvement to~\cite{Kuusi_Mingione_Sire}. Additionally, we do not consider solutions to~\eqref{Eq: Variational zero} but we prove that the same Caccioppoli-type inequality is valid for solutions to the homogeneous resolvent problem
\begin{align*}
 \lambda \langle u , v \rangle_{\L^2} + \fa (u , v) = 0 \qquad (v \in \C_c^{\infty} (B(x_0 , 3 r))),
\end{align*}
where $\lambda \in \S_{\theta} := \{ z \in \IC \setminus \{ 0 \} : \lvert \arg(z) \rvert < \theta \}$ for some $\theta \in (\pi / 2 , \pi)$ depending on $\Lambda$. It is important to note that the constant $C$ in the Caccioppoli-type inequality is uniform with respect to $\lambda$. This is proven in Proposition~\ref{Prop: Caccioppoli} below. This allows us by an application of Sobolev's embedding theorem to establish in Lemma~\ref{Lem: Nonlocal} a non-local weak reverse H\"older estimate for such solutions. In particular, this enables us to apply to each of the resolvent operators $T_{\lambda} := \lambda (\lambda + A)^{-1}$ the following $\L^p$-extrapolation theorem for non-local operators which is proven in Section~\ref{Sec: Shen}.

\begin{theorem}
\label{Thm: Shens Lp extrapolation theorem, Banach space valued}
Let $X$ and $Y$ be Banach spaces, $\cM > 0$, and let 
\begin{align*}
 T \in \Lop(\L^2(\IR^d ; X) , \L^2(\IR^d ; Y)) \quad \text{with} \quad \| T \|_{\Lop(\L^2(\IR^d ; X) , \L^2(\IR^d ; Y))} \leq \cM.
\end{align*}

\indent Suppose that there exist constants $p > 2$, $\iota > 1$, and $C > 0$ such that for all balls $B \subset \IR^d$ and all compactly supported $f \in \L^{\infty}(\IR^d ; X)$ with $f = 0$ in $\iota B$ the estimate
\begin{align}
\label{Eq: Weak reverse Hoelder inequality}
 \begin{aligned}
 \bigg( \fint_{B} \| T f \|_Y^p \; \d x \bigg)^{\frac{1}{p}} &\leq C \sup_{B^{\prime} \supset B} \bigg( \fint_{B^{\prime}} \big( \| T f \|_Y^2 + \| f \|_X^2 \big) \; \d x \bigg)^{\frac{1}{2}}
 \end{aligned}
\end{align}
holds. Here the supremum runs over all balls $B^{\prime}$ containing $B$. \par
Then for each $2 < q < p$ the restriction of $T$ onto $\L^2(\IR^d ; X) \cap \L^q(\IR^d ; X)$ extends to a bounded linear operator from $\L^q (\IR^d ; X)$ into $\L^q(\Omega ; Y)$, with operator norm bounded by a constant depending on $d$, $p$, $q$, $\iota$, $C$, and $\cM$.
\end{theorem}

As an application of Theorem~\ref{Thm: Shens Lp extrapolation theorem, Banach space valued} we get the following result.

\begin{theorem}
\label{Thm: Resolvent}
Let $d \geq 1$, $\alpha \in (0 , 1)$, and $K : \IR^d \times \IR^d \to \IC$ be subject to~\eqref{Eq: Ellipticity} for some $0 < \Lambda < 1$. Let $A$ denote the operator associated to the sesquilinear form~\eqref{Eq: Sesquilinear form}. Then, for $\Phi := \pi - \arccos (\Lambda^2)$ one has that $\S_{\Phi}$ is contained in the resolvent set of $- A$. Moreover, for each $\theta \in (0 , \Phi)$ there exists $\eps > 0$ such that for all numbers $p$ satisfying
\begin{align*}
 \Big\lvert \frac{1}{p} - \frac{1}{2} \Big\rvert < \frac{\alpha}{d} + \eps
\end{align*}
and for all $\lambda \in \S_{\theta}$ the restriction of the resolvent operator $(\lambda + A)^{-1}$ onto $\L^2 (\IR^d) \cap \L^p (\IR^d)$ extends to a bounded operator on $\L^p (\IR^d)$. In particular, there exists a constant $C > 0$ such that for all $\lambda \in \S_{\theta}$ and all $f \in \L^2 (\IR^d) \cap \L^p (\IR^d)$ the inequality
\begin{align}
\label{Eq: Resolvent estimate}
 \| \lambda (\lambda + A)^{-1} f \|_{\L^p (\IR^d)} \leq C \| f \|_{\L^p (\IR^d)}
\end{align}
holds. Here, the constant $\eps$ depends on $d$, $\theta$, and $\Lambda$ and the constant $C$ depends on $d$, $\theta$, $\Lambda$, and $p$.
\end{theorem}

It is well-known that the resolvent estimate presented in Theorem~\ref{Thm: Resolvent} forms the basis to a rich parabolic theory of the operator $A$. Indeed, since $\Phi > \pi / 2$ the resolvent estimate~\eqref{Eq: Resolvent estimate} is equivalent to the fact that the $\L^p$-realization of $- A$ generates a bounded analytic semigroup $(\e^{- t A})_{t \geq 0}$ on $\L^p (\IR^d)$. \par
An important notion in the parabolic theory is the notion of maximal regularity, cf.~\cite{Denk_Hieber_Pruess, Kunstmann_Weis, Weis}. To this end, consider the Cauchy problem
\begin{align}
\label{Eq: ACP}
\left\{ \begin{aligned}
 u^{\prime} (t) + A u (t) &= f (t), \qquad (t > 0) \\
 u (0) &= 0.
\end{aligned} \right.
\end{align}
Let $1 < r < \infty$ and let $f \in \L^r (0 , \infty ; \L^p (\IR^d))$. The unique mild solution to~\eqref{Eq: ACP} is given by the variation of constants formula
\begin{align*}
 u (t) := \int_0^t \e^{- (t - s) A} f (s) \; \d s \qquad (t > 0).
\end{align*}
We say that $A$ has maximal $\L^r$-regularity if for all $f \in \L^r (0 , \infty ; \L^p (\IR^d))$ one has that
\begin{align*}
 u^{\prime} , A u \in \L^r (0 , \infty ; \L^p (\IR^d)).
\end{align*}
In this situation, it is well-known that the closed graph theorem implies that there exists a constant $C > 0$ such that for all $g \in \L^r (0 , \infty ; \L^p (\IR^d))$ the estimate
\begin{align*}
 \| u^{\prime} \|_{\L^r (0 , \infty ; \L^p (\IR^d))} + \| A u \|_{\L^r (0 , \infty ; \L^p (\IR^d))} \leq C \| f \|_{\L^r (0 , \infty ; \L^p (\IR^d))}
\end{align*}
holds. That $A$ has indeed maximal $\L^r$-regularity is formulated in the last theorem.

\begin{theorem}
\label{Thm: Maximal regularity}
Let $d \geq 1$, $\alpha \in (0 , 1)$, and $K : \IR^d \times \IR^d \to \IC$ be subject to~\eqref{Eq: Ellipticity} for some $0 < \Lambda < 1$. Let $A$ denote the operator associated to the sesquilinear form~\eqref{Eq: Sesquilinear form}. Then there exists $\eps > 0$ such that for all numbers $p$ satisfying
\begin{align*}
 \Big\lvert \frac{1}{p} - \frac{1}{2} \Big\rvert < \frac{\alpha}{d} + \eps
\end{align*}
and for all $1 < r < \infty$ the operator $A$ has maximal $\L^r$-regularity.
\end{theorem}

For the fractional Laplacian, i.e., if the kernel $K$ satisfies the $2 K(x , y) = C_{d , \alpha} / \lvert x - y \rvert^{d + 2 \alpha}$, a similar theorem on finite time intervals was proven by Biccari, Warma, and Zuazua~\cite{Biccari_Warma_Zuazua_2}. \par
We shortly outline the structure of the paper. As mentioned above, Theorem~\ref{Thm: Shens Lp extrapolation theorem, Banach space valued} is proven in Section~\ref{Sec: Shen}. Section~\ref{Sec: Caccioppoli} is reserved to prove the Caccioppoli-type estimate and the non-local weak reverse H\"older estimate. In the final Section~\ref{Sec: Proofs} we prove Theorems~\ref{Thm: Resolvent} and~\ref{Thm: Maximal regularity}.

\section{An $\L^p$-extrapolation theorem for non-local operators}
\label{Sec: Shen}

\noindent This section is devoted to prove a non-local version of the $\L^p$-extrapolation theorem of Shen~\cite[Thm.~3.1]{Shen}. The proof follows Shen's original argument and is only modified slightly. However, for the convenience of the reader, we present the complete argument here. The proof carries out a good-$\lambda$ argument and bases on the following version of the Calder\'on--Zygmund decomposition which was proven by Caffarelli and Peral~\cite[Lem.~1.1]{Caffarelli_Peral}.

\begin{lemma}
 \label{Lem: Calderon-Zygmund decomposition}
Let $Q$ be a bounded cube in $\IR^d$ and $\cA \subset Q$ a measurable set satisfying
\begin{align*}
 0 < \lvert \cA \rvert < \delta \lvert Q \rvert \quad \text{for some} \quad 0 < \delta < 1.
\end{align*}
Then there is a family of disjoint dyadic cubes $\{Q_k\}_{k \in \IN}$ obtained by suitable selections of successive bisections of $Q$, such that for all $k \in \IN$
\begin{align*}
 \text{a) } \lvert \cA \setminus \bigcup_{l \in \IN} Q_l \rvert = 0, \qquad \text{b) } \lvert \cA \cap Q_k \rvert > \delta \lvert Q_k \rvert, \qquad \text{c) } \lvert \cA \cap Q_k^* \rvert \leq \delta \lvert Q_k^* \rvert,
\end{align*}
where $Q_k^*$ is the dyadic parent of $Q_k$.
\end{lemma}

Here and in the following, we denote for a ball $B \subset \IR^d$ and a cube $Q \subset \IR^d$ the by the factor $\iota > 0$ dilated ball and cube with the same center by $\iota B$ and $\iota Q$.

\begin{proof}
For this proof, we denote a generic constant that depends solely on $d$, $p$, $q$, $\iota$, or the constant $C$ in inequality~\eqref{Eq: Weak reverse Hoelder inequality} by $C_g$. We abbreviate the operator norm $\|T\|_{\Lop(\L^2(\Omega ; X) , \L^2(\Omega ; Y))}$ by $\| T \|$. \par
Let $x_0 \in \IR^d$, $r > 0$, and $Q$ be a cube in $\IR^d$ with $\diam(Q) = 2 r$ and midpoint $x_0$. Let $B := B(x_0 , r)$. One directly verifies that
\begin{align*}
 \frac{1}{\sqrt{d}} B \subset Q \subset B \quad \text{and} \quad \abs{Q} = \Big( \frac{2 r}{\sqrt{d}} \Big)^d.
\end{align*}
Thus, without loss of generality we can assume that estimate~\eqref{Eq: Weak reverse Hoelder inequality} is valid for cubes centered in $x_0$ instead of balls and for some possibly different $\iota$. \par
Fix $q \in (2 , p)$ and take $f \in \L^{\infty}(\IR^d ; X)$ with compact support. For $\lambda > 0$ consider the set
\begin{align*}
 E(\lambda) := \{x \in \IR^d : \M (\| T f \|_Y^2 )(x) > \lambda \},
\end{align*}
where $\M$ denotes the Hardy--Littlewood maximal operator. Since $\|T f\|_Y^2 \in \L^1(\IR^d)$, the weak-type estimate of the maximal operator~\cite[Thm.~I.1]{Stein} implies
\begin{align}
\label{Eq: Weak L1 estimate}
\lvert E(\lambda) \rvert \leq \frac{C_g}{\lambda} \big\| \| T f \|_Y^2 \big\|_{\L^1(\IR^d)} = \frac{C_g}{\lambda} \|T f\|_{\L^2(\IR^d ; Y)}^2.
\end{align}
Let $A := 1 / (2 \delta^{2 / q}) > 5^d$, where $\delta \in (0 , 1)$ is a small constant to be determined. Decompose $\IR^d$ into a dyadic grid. Then by~\eqref{Eq: Weak L1 estimate} we find a mesh size such that each cube $Q_0$ from the grid satisfies
\begin{align*}
 \abs{E(A \lambda)} < \delta \abs{Q_0}.
\end{align*}
Note that the mesh size is allowed to depend on $\lambda$, $\delta$, $f$, and $T$. If the case $\abs{Q_0 \cap E(A \lambda)} = 0$ applies, do nothing. In the other case, the set defined by $\cA := Q_0 \cap E(A \lambda)$ together with the cube $Q_0$ satisfy the assumptions of Lemma~\ref{Lem: Calderon-Zygmund decomposition}. Proceeding in that way for every cube $Q_0$ in the grid and enumerating all cubes obtained in this way by Lemma~\ref{Lem: Calderon-Zygmund decomposition} by $\{ Q_k \}_{k \in \IN}$ yields a countable family of mutually disjoint cubes satisfying for all $k \in \IN$
\begin{align*}
 (\mathrm{i})~\lvert E(A \lambda) \setminus \bigcup_{l \in \IN} Q_l \rvert = 0, \quad (\mathrm{ii})~\lvert E(A \lambda) \cap Q_k \rvert > \delta \lvert Q_k \rvert, \quad \text{and} \quad (\mathrm{iii})~\lvert E(A \lambda) \cap Q_k^* \rvert \leq \delta \lvert Q_k^* \rvert.
\end{align*}
Note that as in Lemma \ref{Lem: Calderon-Zygmund decomposition}, $Q_k^*$ denotes the dyadic parent of $Q_k$.

\subsection*{Claim 1} The operator $T$ is $\L^q$-bounded, once there are constants $\delta , \gamma > 0$ such that for all $\lambda > 0$
\begin{align}
\label{Eq: Decomposition of Distribution-function}
 \lvert E(A \lambda) \rvert \leq \delta \lvert E(\lambda) \rvert + \lvert \{x \in \IR^d : \M(\| f \|_X^2)(x) > \lambda \gamma \} \rvert
\end{align}
holds. \par
To see this, first note that~\eqref{Eq: Weak L1 estimate} and $q > 2$ imply that the function $\lambda \mapsto \lambda^{q / 2 - 1} \lvert E(\lambda) \rvert$ is in $\Lloc^1([0 , \infty))$. The premise of Claim~1 and the definition of $A$ imply that for all $\lambda_0 > 0$
\begin{align*}
 \int_0^{A \lambda_0} \lambda^{\frac{q}{2} - 1} \lvert E(\lambda) \rvert \; \d \lambda &\leq \delta \int_0^{A \lambda_0} \lambda^{\frac{q}{2} - 1} \lvert E(A^{-1} \lambda) \rvert \; \d \lambda \\
 &\quad+ \int_0^{A \lambda_0} \lambda^{\frac{q}{2} - 1} \lvert \{x \in \IR^d : \M(\| f \|_X^2)(x) > A^{-1} \lambda \gamma \} \rvert \; \d \lambda \\
 &\leq 2^{-\frac{q}{2}}  \int_0^{\lambda_0} \lambda^{\frac{q}{2} - 1} \lvert E(\lambda) \rvert \; \d \lambda \\
 &\quad+ \int_0^{\infty} \lambda^{\frac{q}{2} - 1} \lvert \{x \in \IR^d : \M(\| f \|_X^2)(x) > A^{-1} \lambda \gamma \} \rvert \; \d \lambda.
\end{align*}
Perform a linear transformation in the second integral and notice that the resulting integral coincides modulo a factor by a generic constant times $\delta \gamma^{- q / 2}$ with $\| \M (\| f \|_X^2) \|_{\L^{q / 2} (\IR^d)}^{q / 2}$, cf.~\cite[Prop.~1.1.4]{Grafakos}. By virtue of the boundedness of the maximal operator on $\L^{q / 2}$~\cite[Thm.~I.1]{Stein} this results in the estimate
\begin{align*}
\int_0^{A \lambda_0} \lambda^{\frac{q}{2} - 1} \lvert E(\lambda) \rvert \; \d \lambda \leq 2^{-\frac{q}{2}} \int_0^{\lambda_0} \lambda^{\frac{q}{2} - 1} \lvert E(\lambda) \rvert \; \d \lambda + \delta \gamma^{- \frac{q}{2}} C_g \| f \|_{\L^q(\IR^d ; X)}^q.
\end{align*}
Using that $A > 1$, the first term on the right-hand side can be absorbed by the left-hand side. This yields
\begin{align*}
 \int_0^{A \lambda_0} \lambda^{\frac{q}{2} - 1} \lvert E(\lambda) \rvert \; \d \lambda \leq \delta \gamma^{- \frac{q}{2}} C_g \| f \|_{\L^q(\IR^d ; X)}^q.
\end{align*}
Taking $\lambda_0 \to \infty$ and using $\| [T f](x) \|_Y^2 \leq \M(\| T f \|_Y^2)(x)$ for almost every $x \in \IR^d$ yields together with~\cite[Prop.~1.1.4]{Grafakos} that
\begin{align*}
 \| T f \|_{\L^q(\IR^d ; Y)}^q \leq \gamma^{- \frac{q}{2}} C_g \| f \|_{\L^q(\IR^d ; X)}^q.
\end{align*}
The conclusion of the theorem follows by density (note that simple functions with bounded support are dense in all $\L^q(\Omega ; X)$-spaces by construction of the Bochner integral).

\subsection*{Claim 2} The premise of Claim~$1$ follows if there are constants $\delta , \gamma > 0$ such that for all dyadic parents $Q_k^*$ of the family of cubes $\{ Q_k \}_{k \in \IN}$ constructed before $\mathrm{(i)}$-$\mathrm{(iii)}$ the following statement is valid:
\begin{align*}
 Q_k^* \cap \{x \in \IR^d : \M(\| f \|_X^2)(x) \leq \lambda \gamma \} \neq \emptyset \quad \text{implies} \quad Q_k^* \subset E(\lambda).
\end{align*}
\indent Since~\eqref{Eq: Decomposition of Distribution-function} is trivial if $\lvert E(A \lambda) \rvert = 0$ assume that $\lvert E(A \lambda) \rvert > 0$. Let $I \subset \IN$ be the index set of all $l \in \IN$ such that $\{ Q_l^* \}_{l \in I}$ is a maximal set of mutually disjoint cubes satisfying $Q_l^* \cap \{x \in \IR^d : \M(\| f \|_X^2)(x) \leq \lambda \gamma \} \neq \emptyset$. Then,
\begin{align*}
 \lvert E(A \lambda) \rvert &= \lvert E(A \lambda) \cap \{x \in \IR^d : \M (\| f \|_X^2)(x) \leq \lambda \gamma \} \rvert \\
 &\qquad+ \lvert E(A \lambda) \cap \{x \in \IR^d : \M (\| f \|_X^2)(x) > \lambda \gamma \} \rvert.
\end{align*}
Dealing the first term on the right-hand side by the maximality of $\{ Q_l^* \}_{l \in I}$ together with $\mathrm{(i)}$ and the second term by using the monotonicity of the Lebesgue measure yields
\begin{align*}
 \lvert E(A \lambda) \rvert \leq \sum_{l \in I} \lvert E(A \lambda) \cap Q_l^* \rvert + \lvert \{x \in \IR^d : \M (\| f \|_X^2)(x) > \lambda \gamma \} \rvert.
\end{align*}
Next, use $\mathrm{(iii)}$ first and then the mutual disjointness of the family $\{ Q_l^* \}_{l \in I}$ together with the assertion of Claim~2 to get
\begin{align*}
 \lvert E(A \lambda) \rvert \leq \delta \lvert E(\lambda) \rvert +\lvert \{x \in \IR^d : \M (\| f \|_X^2)(x) > \lambda \gamma \} \rvert.
\end{align*}

\subsection*{Claim 3} There exist $\delta , \gamma > 0$ such that
\begin{align*}
 Q_k^* \cap \{x \in \IR^d : \M (\| f \|_X^2)(x) \leq \lambda \gamma \} \neq \emptyset \quad \text{implies} \quad Q_k^* \subset E(\lambda).
\end{align*}
\indent To conclude this statement, we argue by contradiction. For this purpose, suppose that there exists a $Q_k$ with $\{ x \in Q_k^* : \M (\| f \|_X^2)(x) \leq \gamma \lambda \} \neq \emptyset$ and $Q_k^* \setminus E(\lambda) \neq \emptyset$. We show, that the existence of such a cube contradicts~$\mathrm{(ii)}$. In this situation, for every cube $Q$ that contains $Q_k^*$, we have
\begin{align}
\label{Eq: Mean value estimates for contradiction}
  \frac{1}{\lvert Q \rvert} \int_Q \| f \|_X^2 \; \d x \leq \gamma \lambda \qquad \text{and} \qquad \frac{1}{\lvert Q \rvert} \int_Q \| T f \|_Y^2 \; \d x \leq \lambda.
\end{align}
Next, let $x \in Q_k$ and $Q^{\prime}$ be a cube with $x \in Q^{\prime}$ and $Q^{\prime} \nsubset 2 Q_k^*$. Then, we find for the sidelength of $Q^{\prime}$ that $\ell(Q^{\prime}) > \ell(Q_k)$. If $y \in Q_k^*$, $1 \leq i \leq d$, and if $x^{\prime}$ denotes the center of $Q^{\prime}$, then
\begin{align*}
 \abs{y_i - x^{\prime}_i} &\leq \abs{y_i - x_i} + \abs{x_i - x^{\prime}_i} \leq 2 \ell(Q_k) + \frac{1}{2} \ell(Q^{\prime}) < \frac{5}{2} \ell(Q^{\prime}).
\end{align*}
Consequently, we have $Q_k^* \subset 5 Q^{\prime}$ and thus by virtue of~\eqref{Eq: Mean value estimates for contradiction} we have for $x \in Q_k$
\begin{align*}
 \M ( \| T f \|_Y^2) (x) &= \max\bigg\{ \M_{2 Q_k^*} (\| T f \|_Y^2) (x) , \sup_{\substack{Q^{\prime} \ni x \\ Q^{\prime} \nsubset 2 Q_k^*}} \frac{1}{\lvert Q^{\prime} \rvert} \int_{Q^{\prime}} \| T f \|_Y^2 \; \d y \bigg\} \\
 &\leq \max \{ \M_{2 Q_k^*} (\| T f \|_Y^2) (x) , 5^d \lambda \},
\end{align*}
where $\M_{2 Q_k^*}$ denotes the localized maximal operator
\begin{align*}
 (\M_{2 Q_k^*} g) (x) := \sup_{\substack{x \in R \subset 2 Q_k^* \\ R \text{ cube}}} \fint_R \lvert g (y) \rvert \; \d y \qquad (x \in 2 Q_k^*).
\end{align*}
Since $A = 1 / (2 \delta^{2 / q}) > 5^d$, we derive
\begin{align*}
 \lvert E(A \lambda) \cap Q_k \rvert \leq \lvert \{ x \in Q_k : \M_{2 Q_k^*} (\| T f \|_Y^2)(x) > A \lambda \} \rvert.
\end{align*}
Use $(a + b)^2 \leq 2(a^2 + b^2)$ together with~\cite[Prop.~1.1.3]{Grafakos} to estimate
\begin{align*}
 \lvert E(A \lambda) \cap Q_k \rvert &\leq \Big\lvert \Big\{ x \in Q_k : \M_{2 Q_k^*} ( \| T (f \chi_{2 \iota Q_k^*}) \|_Y^2 )(x) > \frac{A \lambda}{4} \Big\} \Big\rvert \\
 &\qquad+ \Big\lvert \Big\{ x \in Q_k : \M_{2 Q_k^*} ( \| T (f \chi_{\IR^d \setminus 2 \iota Q_k^*}) \|_Y^2)(x) > \frac{A \lambda}{4} \Big\} \Big\rvert \\
 &=: \cA + \cB.
\end{align*}
By means of the weak-type estimate of $\M_{2 Q_k^*}$ in the first inequality below, and the boundedness of $T$ from $\L^2(\IR^d ; X)$ into $\L^2(\IR^d ; Y)$ together with~\eqref{Eq: Mean value estimates for contradiction} in the second inequality below, we derive
\begin{align}
\label{Eq: Estimate of A}
 \cA \leq \frac{C_g}{A \lambda} \int_{2 Q_k^*} \| T (f \chi_{2 \iota Q_k^*}) \|_Y^2 \; \d x \leq \lvert Q_k \rvert \|T\|^2 \frac{C_g \gamma}{A}.
\end{align}
Next, the continuous embedding $\L^{p / 2}(2Q_k^*) \subset \L^{p / 2 , \infty}(2 Q_k^*)$ and the $\L^{p / 2}$-boundedness of $\M_{2 Q_k^*}$ yield
\begin{align*}
 (A \lambda)^{\frac{p}{2}} \cB \leq C_g \big\| \M_{2 Q_k^*} (\| T (f \chi_{\IR^d \setminus 2 \iota Q_k^*}) \|_Y^2) \big\|_{\L^{p / 2}(2 Q_k^*)}^{\frac{p}{2}} \leq C_g \int_{2 Q_k^*} \| T (f \chi_{\IR^d \setminus 2 \iota Q_k^*}) \|_Y^p \; \d x. 
\end{align*}
An application of~\eqref{Eq: Weak reverse Hoelder inequality} yields
\begin{align*}
 (A \lambda)^{\frac{p}{2}} \cB &\leq \lvert Q_k \rvert C_g \bigg\{ \sup_{Q^{\prime} \supset 2 Q_k^*} \bigg( \frac{1}{\lvert Q^{\prime} \rvert} \int_{Q^{\prime}} \big( \| T (f \chi_{\IR^d \setminus 2 \iota Q_k^*}) \|_Y^2 + \| f \chi_{\IR^d \setminus 2 \iota Q_k^*} \|_X^2 \big) \; \d x \bigg)^{\frac{1}{2}} \bigg\}^p.
\end{align*}
Add and subtract $f \chi_{2 \iota Q_k^*}$ in the argument of $T$ and use $\|f \chi_{\IR^d \setminus 2 \iota Q_k^*}\|_X \leq \| f \|_X$ together with~\eqref{Eq: Mean value estimates for contradiction} to obtain
\begin{align*}
 (A \lambda)^{\frac{p}{2}} \cB \leq \lvert Q_k \rvert C_g \bigg\{ \sup_{Q^{\prime} \supset 2 Q_k^*} \bigg( \frac{1}{\lvert Q^{\prime} \rvert} \int_{Q^{\prime}} \| T (f \chi_{2 \iota Q_k^*}) \|_Y^2 \; \d x \bigg)^{\frac{1}{2}} + ((\gamma + 1) \lambda)^{\frac{1}{2}} \bigg\}^p.
\end{align*}
Use the $\L^2$-boundedness of $T$ to get
\begin{align*}
 (A \lambda)^{\frac{p}{2}} \cB \leq \lvert Q_k \rvert C_g \bigg\{ \| T \| \sup_{Q^{\prime} \supset 2 Q_k^*} \bigg( \frac{1}{\lvert Q^{\prime} \rvert} \int_{2 \iota Q_k^*} \| f \|_X^2 \; \d x \bigg)^{\frac{1}{2}} + ((\gamma + 1) \lambda)^{\frac{1}{2}} \bigg\}^p.
\end{align*}
Notice that $\lvert 2 \iota Q_k^* \rvert / \lvert Q^{\prime} \rvert \leq \iota^d$ so that another application of~\eqref{Eq: Mean value estimates for contradiction} finally yields
\begin{align}
\label{Eq: Estimate of B}
 (A \lambda)^{\frac{p}{2}} \cB \leq \lvert Q_k \rvert C_g \lambda^{\frac{p}{2}} \big\{ \| T \| \gamma^{\frac{1}{2}} + (\gamma + 1)^{\frac{1}{2}} \big\}^p.
\end{align}
Recall that $A = 1 / (2 \delta^{2 / q})$ and that $\| T \| \leq \cM$, so that~\eqref{Eq: Estimate of A} and~\eqref{Eq: Estimate of B} yields
\begin{align*}
 \lvert E(A \lambda) \cap Q_k \rvert &\leq \cA + \cB \\
 &\leq C_g \delta \lvert Q_k \rvert \bigg\{ \frac{\gamma \cM^2}{A \delta} + \frac{\big\{ \cM \gamma^{\frac{1}{2}} + (\gamma + 1)^{\frac{1}{2}} \big\}^p}{A^{\frac{p}{2}} \delta} \bigg\} \\
 &\leq C_g \delta \lvert Q_k \rvert \big\{ \gamma \delta^{\frac{2}{q} - 1} \cM^2 + \big\{ \cM \gamma^{\frac{1}{2}} + (\gamma + 1)^{\frac{1}{2}} \big\}^p \delta^{\frac{p}{q} - 1} \big\}.
\end{align*}
Since $p > q$, we can choose $\delta$ small enough such that $\{ \cM + \sqrt{2} \}^p \delta^{\frac{p}{q} - 1} \leq 1 / (2 C_g)$. For this fixed value of $\delta$ choose $\gamma \leq \min\{ 1 , \delta^{1 - \frac{2}{q}} / (2 \cM^2 C_g) \}$. Then, we obtain $\lvert E(A \lambda) \cap Q_k \rvert \leq \delta \lvert Q_k \rvert$ which is a contradiction to $\mathrm{(ii)}$ of the Calder\'{o}n--Zygmund decomposition.
\end{proof}

\section{Verification of the non-local weak reverse H\"older estimate}
\label{Sec: Caccioppoli}

\noindent The main ingredient for the verification of the non-local weak reverse H\"older estimate is the following proposition. In this proposition a non-local Caccioppoli inequality for the resolvent equation is proved. The prove follows the ideas of the proof for the special case $\lambda = 0$ and for real-valued kernel functions $K(x , y)$ which can be found in the work of Kuusi, Mingione, and Sire~\cite[Thm.~3.2]{Kuusi_Mingione_Sire}.

\begin{proposition}
\label{Prop: Caccioppoli}
Let $x_0 \in \IR^d$, $r > 0$, and $\alpha \in (0 , 1)$. Let further $\theta \in (0 , \Phi)$, where $\Phi := \pi - \arccos (\Lambda^2)$ and let $u \in \H^{\alpha} (\IR^d)$ be a function that satisfies for all $v \in \C_c^{\infty} (B(x_0 , 3 r / 2))$ the equation
\begin{align}
\label{Eq: Variational}
 \lambda \int_{\IR^d} u (x) \overline{v (x)} \; \d x + \int_{\IR^d} \int_{\IR^d} K(x , y) (u (x) - u (y)) \overline{(v (x) - v (y))} \; \d x \; \d y = 0.
\end{align}
Let $\eta \in \C_c^{\infty} (B(x_0 , 3 r / 2))$ be a function that satisfies $0 \leq \eta \leq 1$ and $\| \nabla \eta \|_{\L^{\infty}} \leq C_d / r$ for some constant $C_d > 0$ depending only on $d$. Then there exists a constant $C > 0$ depending only on $d$, $\alpha$, $\Lambda$, and $\theta$ such that
\begin{align*}
 \lvert \lambda \rvert \int_{\IR^d} \lvert u (x) \rvert^2 \eta (x)^2 \; \d x &+ \int_{2 B} \int_{2 B} \frac{\lvert u (x) - u (y) \rvert^2 (\eta (x)^2 + \eta (y)^2)}{\lvert x - y \rvert^{d + 2 \alpha}} \; \d x \; \d y \\
 &+ \int_{2 B} \int_{2 B} \frac{\lvert u (x) \eta (x) - u (y) \eta (y) \rvert^2}{\lvert x - y \rvert^{d + 2 \alpha}} \; \d x \; \d y \\
 &+ \int_{2 B} \int_{\IR^d \setminus 2 B} \frac{\lvert u (y) \rvert^2 \eta(y)^2}{\lvert x - y \rvert^{d + 2 \alpha}} \; \d x \; \d y + \int_{\IR^d \setminus 2 B} \int_{2 B} \frac{\lvert u (x) \rvert^2 \eta (x)^2}{\lvert x - y \rvert^{d + 2 \alpha}} \; \d x \; \d y \\
 &\qquad \leq C \bigg( \frac{1}{r^{2 \alpha}} \int_{2 B} \lvert u (x) \rvert^2 \; \d x + \int_{2 B} \lvert u (x) \rvert \; \d x \int_{\IR^d \setminus 2 B} \frac{\lvert u (y) \rvert}{\lvert x_0 - y \rvert^{d + 2 \alpha}} \; \d y \bigg).
\end{align*}
\end{proposition}

\begin{proof}
Write $B := B(x_0 , r)$. Since $\C_c^{\infty} (\IR^d)$ is dense in $\H^{\alpha} (\IR^d)$, it is possible to choose $v := u \eta^2$ as a test function. Thus, by virtue of~\eqref{Eq: Variational} the following identity holds
\begin{align}
\label{Eq: Tested by u}
 \lambda \int_{\IR^d} \lvert u (x) \rvert^2 \eta (x)^2 \; \d x + \int_{\IR^d} \int_{\IR^d} K (x , y) (u (x) - u (y)) \overline{(u (x) \eta (x)^2 - u(y) \eta(y)^2)} \; \d x \; \d y = 0.
\end{align}
Decompose the double integral into
\begin{align*}
 &\int_{\IR^d} \int_{\IR^d} K (x , y) (u (x) - u (y)) \overline{(u (x) \eta (x)^2 - u(y) \eta(y)^2)} \; \d x \; \d y \\
 &\qquad = \int_{2 B} \int_{2 B} K (x , y) (u (x) - u (y)) \overline{(u (x) \eta (x)^2 - u(y) \eta(y)^2)} \; \d x \; \d y \\
 &\qquad\qquad - \int_{2 B} \int_{\IR^d \setminus 2 B} K (x , y) (u (x) - u (y)) \overline{u(y)} \eta(y)^2 \; \d x \; \d y \\
 &\qquad\qquad + \int_{\IR^d \setminus 2 B} \int_{2 B} K (x , y) (u (x) - u (y)) \overline{u (x)} \eta (x)^2 \; \d x \; \d y.
\end{align*}
After rearranging the terms, we find by~\eqref{Eq: Tested by u} that
\begin{align}
\label{Eq: Ordering terms}
\begin{aligned}
 \lambda \int_{\IR^d} \lvert u (x) \rvert^2 \eta (x)^2 \; \d x &+ \int_{2 B} \int_{2 B} K (x , y) (u (x) - u (y)) \overline{(u (x) \eta (x)^2 - u(y) \eta(y)^2)} \; \d x \; \d y \\
 &+ \int_{2 B} \int_{\IR^d \setminus 2 B} K (x , y) \lvert u (y) \rvert^2 \eta(y)^2 \; \d x \; \d y \\
 &+ \int_{\IR^d \setminus 2 B} \int_{2 B} K (x , y) \lvert u (x) \rvert^2\eta (x)^2 \; \d x \; \d y \\
 &= \int_{2 B} \int_{\IR^d \setminus 2 B} K (x , y) u (x) \overline{u(y)} \eta(y)^2 \; \d x \; \d y \\
 &\qquad - \int_{\IR^d \setminus 2 B} \int_{2 B} K (x , y) u (y) \overline{u (x)} \eta (x)^2 \; \d x \; \d y.
\end{aligned}
\end{align}
To rewrite the second term on the left-hand side of~\eqref{Eq: Ordering terms}, calculate
\begin{align}
\label{Eq: Auxiliary calculation second term}
\begin{aligned}
 &(u (x) - u (y)) \overline{(u (x) \eta (x)^2 - u (y) \eta (y)^2)} \\
 &\qquad = \lvert u (x) - u (y) \rvert^2 \eta (x)^2 + (u (x) - u (y)) \overline{u (y)} (\eta (x) - \eta (y)) (\eta(x) + \eta (y)).
 \end{aligned}
\end{align}
Switch the roles of $x$ and $y$, perform the same calculation as in~\eqref{Eq: Auxiliary calculation second term}, and add the resulting identity to~\eqref{Eq: Auxiliary calculation second term} to obtain
\begin{align}
\label{Eq: Auxiliary calculation second term, 2}
\begin{aligned}
 &2 (u (x) - u (y)) \overline{(u (x) \eta (x)^2 - u (y) \eta (y)^2)} \\
 &= \lvert u (x) - u (y) \rvert^2 (\eta (x)^2 + \eta (y)^2) + (u (x) - u (y)) (\overline{u(x)} + \overline{u (y)}) (\eta (x) - \eta (y)) (\eta (x) + \eta(y)).
\end{aligned}
\end{align}
Now, plug~\eqref{Eq: Auxiliary calculation second term, 2} into~\eqref{Eq: Ordering terms} and rearrange terms to obtain
\begin{align}
\label{Eq: Correctly ordered}
\begin{aligned}
  &\lambda \int_{\IR^d} \lvert u (x) \rvert^2 \eta (x)^2 \; \d x \\
  &+ \frac{1}{2} \int_{2 B} \int_{2 B} K (x , y) \lvert u (x) - u (y) \rvert^2 (\eta (x)^2 + \eta (y)^2) \; \d x \; \d y \\
 &+ \int_{2 B} \int_{\IR^d \setminus 2 B} K (x , y) \lvert u (y) \rvert^2 \eta(y)^2 \; \d x \; \d y \\
 &+ \int_{\IR^d \setminus 2 B} \int_{2 B} K (x , y) \lvert u (x) \rvert^2\eta (x)^2 \; \d x \; \d y \\
 &= \int_{2 B} \int_{\IR^d \setminus 2 B} K (x , y) u (x) \overline{u(y)} \eta(y)^2 \; \d x \; \d y \\
 &\qquad - \int_{\IR^d \setminus 2 B} \int_{2 B} K (x , y) u (y) \overline{u (x)} \eta (x)^2 \; \d x \; \d y \\
 &\qquad - \frac{1}{2} \int_{2 B} \int_{2 B} K (x , y) (u (x) - u (y)) (\overline{u(x)} + \overline{u (y)}) (\eta (x) - \eta (y)) (\eta (x) + \eta(y)) \; \d x \; \d y.
\end{aligned}
\end{align}

\indent Notice, that on the left-hand side, the complex number $\lambda$ is multiplied by a non-negative real number and that in all other integrals on the left-hand side, the complex-valued function $K (x , y)$ is multiplied by non-negative real functions. The condition~\eqref{Eq: Ellipticity} implies that
\begin{align*}
 \arg(K(x , y)) \in \S_{\pi - \Phi} \quad \text{with} \quad \Phi = \pi - \arccos (\Lambda^2).
\end{align*}
Since $\pi - \Phi < \pi / 2$, $\lambda \in \S_{\theta}$, and $\theta$ is chosen such that $\theta  + (\pi - \Phi) < \pi$, an elementary trigonometric argument shows that there exists a constant $C_{\theta , \Lambda} > 0$ depending only on $\theta$ and $\Lambda$ such that for all $z \in \overline{\S_{\theta}}$ and all $w_1 , w_2 , w_3 \in \overline{\S_{\pi - \Phi}}$ it holds
\begin{align*}
 \lvert z \rvert + \lvert w_1 \rvert + \lvert w_2 \rvert + \lvert w_3 \rvert \leq C_{\theta , \Lambda} \lvert z + w_1 + w_2 + w_3 \rvert.
\end{align*}
Apply this inequality to~\eqref{Eq: Correctly ordered} together with~\eqref{Eq: Ellipticity} to obtain
\begin{align}
\label{Eq: Elliptic estimate}
\begin{aligned}
 &\lvert \lambda \rvert \int_{\IR^d} \lvert u (x) \rvert^2 \eta (x)^2 \; \d x + \frac{\Lambda}{2} \int_{2 B} \int_{2 B} \frac{\lvert u (x) - u (y) \rvert^2 (\eta (x)^2 + \eta (y)^2)}{\lvert x - y \rvert^{d + 2 \alpha}} \; \d x \; \d y \\
 &+ \Lambda \int_{2 B} \int_{\IR^d \setminus 2 B} \frac{\lvert u (y) \rvert^2 \eta(y)^2}{\lvert x - y \rvert^{d + 2 \alpha}} \; \d x \; \d y + \Lambda \int_{\IR^d \setminus 2 B} \int_{2 B} \frac{\lvert u (x) \rvert^2 \eta (x)^2}{\lvert x - y \rvert^{d + 2 \alpha}} \; \d x \; \d y \\
 &\leq C_{\theta , \Lambda} \bigg( \Lambda^{-1} \int_{2 B} \int_{\IR^d \setminus 2 B} \frac{\lvert u (x) \rvert \lvert u(y)\rvert \eta(y)^2}{\lvert x - y \rvert^{d + 2 \alpha}} \; \d x \; \d y + \Lambda^{-1} \int_{\IR^d \setminus 2 B} \int_{2 B} \frac{\lvert u (y) \rvert \lvert u (x)\rvert \eta (x)^2}{\lvert x - y \rvert^{d + 2 \alpha}} \; \d x \; \d y \\
 &\qquad + \frac{\Lambda^{-1}}{2} \int_{2 B} \int_{2 B} \frac{\lvert u (x) - u (y) \rvert (\lvert u(x)\rvert + \lvert u (y) \rvert) \lvert \eta (x) - \eta (y) \rvert (\eta (x) + \eta(y))}{\lvert x - y \rvert^{d + 2 \alpha}} \; \d x \; \d y \bigg).
 \end{aligned}
\end{align}

\indent First, we will rewrite the second term on the left-hand side of~\eqref{Eq: Elliptic estimate}. To this end, use $\lvert z + w \rvert^2 = \lvert z \rvert^2 + \lvert w \rvert^2 + 2 \Re (z \overline{w})$ for $z , w \in \IC$ and calculate
\begin{align}
\label{Eq: Auxiliary calculation}
\begin{aligned}
 \lvert u (x) \eta (x) - u (y) \eta (y) \rvert^2 &= \lvert (u(x) - u (y)) \eta(x) + u (y) (\eta (x) - \eta (y)) \rvert^2 \\
 &= \lvert u (x) - u (y) \rvert^2 \eta (x)^2 + \lvert u (y) \rvert^2 \lvert \eta (x) - \eta (y) \rvert^2 \\
 &\qquad + 2 \Re ([u (x) - u (y)] \overline{u (y)}) \eta(x) (\eta (x) - \eta (y)).
 \end{aligned}
\end{align}
Now, switch the roles of $x$ and $y$, perform the same calculation as in~\eqref{Eq: Auxiliary calculation}, and add the resulting identity to~\eqref{Eq: Auxiliary calculation} to obtain
\begin{align}
\label{Eq: Auxiliary calculation, 2}
\begin{aligned}
 &2 \lvert u (x) \eta (x) - u (y) \eta (y) \rvert^2 \\
 &= \lvert u (x) - u (y) \rvert^2 (\eta (x)^2 + \eta (y)^2) + (\lvert u (x) \rvert^2 + \lvert u (y) \rvert^2) \lvert \eta (x) - \eta (y) \rvert^2 \\
 &\qquad + 2 \Re ([u (x) - u (y)] (\overline{u (y)} \eta(x) + \overline{u (x)} \eta (y))) (\eta (x) - \eta (y)).
\end{aligned}
\end{align}
Notice, that the first term on the right-hand side of~\eqref{Eq: Auxiliary calculation, 2} appears in the second term on the left-hand side of~\eqref{Eq: Elliptic estimate}. Replace half of the second term on the left-hand side of~\eqref{Eq: Elliptic estimate} by employing~\eqref{Eq: Auxiliary calculation, 2} and leave the other half as it is. After rearranging the terms one gets
\begin{align}
\label{Eq: Preparation for Young}
\begin{aligned}
  &\lvert \lambda \rvert \int_{\IR^d} \lvert u (x) \rvert^2 \eta (x)^2 \; \d x + \frac{\Lambda}{4} \int_{2 B} \int_{2 B} \frac{\lvert u (x) - u (y) \rvert^2 (\eta (x)^2 + \eta (y)^2)}{\lvert x - y \rvert^{d + 2 \alpha}} \; \d x \; \d y \\
 &+ \frac{\Lambda}{2} \int_{2 B} \int_{2 B} \frac{\lvert u (x) \eta (x) - u (y) \eta (y) \rvert^2}{\lvert x - y \rvert^{d + 2 \alpha}} \; \d x \; \d y \\
 &+ \Lambda \int_{2 B} \int_{\IR^d \setminus 2 B} \frac{\lvert u (y) \rvert^2 \eta(y)^2}{\lvert x - y \rvert^{d + 2 \alpha}} \; \d x \; \d y + \Lambda \int_{\IR^d \setminus 2 B} \int_{2 B} \frac{\lvert u (x) \rvert^2 \eta (x)^2}{\lvert x - y \rvert^{d + 2 \alpha}} \; \d x \; \d y \\
 &\leq C_{\theta , \Lambda} \bigg( \Lambda^{-1} \int_{2 B} \int_{\IR^d \setminus 2 B} \frac{\lvert u (x) \rvert \lvert u(y)\rvert \eta(y)^2}{\lvert x - y \rvert^{d + 2 \alpha}} \; \d x \; \d y + \Lambda^{-1} \int_{\IR^d \setminus 2 B} \int_{2 B} \frac{\lvert u (y) \rvert \lvert u (x)\rvert \eta (x)^2}{\lvert x - y \rvert^{d + 2 \alpha}} \; \d x \; \d y \\
 &\qquad + \frac{\Lambda^{-1}}{2} \int_{2 B} \int_{2 B} \frac{\lvert u (x) - u (y) \rvert (\lvert u(x)\rvert + \lvert u (y) \rvert) \lvert \eta (x) - \eta (y) \rvert (\eta (x) + \eta(y))}{\lvert x - y \rvert^{d + 2 \alpha}} \; \d x \; \d y \bigg) \\
 &\qquad + \frac{\Lambda}{4} \int_{2 B} \int_{2 B} \frac{(\lvert u (x) \rvert^2 + \lvert u (y) \rvert^2) \lvert \eta (x) - \eta (y) \rvert^2}{\lvert x - y \rvert^{d + 2 \alpha}} \; \d x \; \d y \\
 &\qquad + \frac{\Lambda}{2} \int_{2 B} \int_{2 B} \frac{\Re ([u (x) - u (y)] (\overline{u (y)} \eta(x) + \overline{u (x)} \eta (y))) (\eta (x) - \eta (y))}{\lvert x - y \rvert^{d + 2 \alpha}} \; \d x \; \d y.
 \end{aligned}
\end{align}
The first two terms on the right-hand side are estimated by using Fubini's theorem first and second, if $x_0$ denotes the midpoint of $B$, by using that for $x \in \supp(\eta)$ and $y \in \IR^d \setminus 2 B$ one has due to $\lvert x - y \rvert \geq r / 2$
\begin{align*}
 \lvert x_0 - y \rvert \leq \lvert x_0 - x \rvert + \lvert x - y \rvert \leq \frac{3 r}{2} + \lvert x - y \rvert \leq 4 \lvert x - y \rvert.
\end{align*}
This yields
\begin{align*}
 \int_{2 B} \int_{\IR^d \setminus 2 B} \frac{\lvert u (x) \rvert \lvert u(y)\rvert \eta(y)^2}{\lvert x - y \rvert^{d + 2 \alpha}} \; \d x \; \d y &+ \int_{\IR^d \setminus 2 B} \int_{2 B} \frac{\lvert u (y) \rvert \lvert u (x)\rvert \eta (x)^2}{\lvert x - y \rvert^{d + 2 \alpha}} \; \d x \; \d y \\
 &= 2 \int_{\IR^d \setminus 2 B} \int_{2 B} \frac{\lvert u (y) \rvert \lvert u (x)\rvert \eta (x)^2}{\lvert x - y \rvert^{d + 2 \alpha}} \; \d x \; \d y \\
 &\leq 2 \cdot 4^{d + 2 \alpha} \int_{\IR^d \setminus 2 B} \int_{2 B} \frac{\lvert u (y) \rvert \lvert u (x)\rvert \eta (x)^2}{\lvert x_0 - y \rvert^{d + 2 \alpha}} \; \d x \; \d y \\
 &\leq 2 \cdot 4^{d + 2 \alpha} \int_{2 B} \lvert u (x) \rvert \; \d x \int_{\IR^d \setminus 2 B} \frac{\lvert u (y) \rvert}{\lvert x_0 - y \rvert^{d + 2 \alpha}} \; \d y.
\end{align*}
For the third term on the right-hand side of~\eqref{Eq: Preparation for Young} employ Young's inequality together with $(a + b)^2 \leq 2 (a^2 + b^2)$ for $a , b \geq 0$ to deduce
\begin{align}
\label{Eq: Absorbtion 1}
\begin{aligned}
 &\int_{2 B} \int_{2 B} \frac{\lvert u (x) - u (y) \rvert (\lvert u(x)\rvert + \lvert u (y) \rvert) \lvert \eta (x) - \eta (y) \rvert (\eta (x) + \eta(y))}{\lvert x - y \rvert^{d + 2 \alpha}} \; \d x \; \d y \\
 &\leq \frac{\Lambda^2}{8 C_{\theta , \Lambda}} \int_{2 B} \int_{2 B} \frac{\lvert u (x) - u (y) \rvert^2 (\eta (x)^2 + \eta(y)^2)}{\lvert x - y \rvert^{d + 2 \alpha}} \; \d x \; \d y \\
 &\qquad + \frac{8 C_{\theta , \Lambda}}{\Lambda^2} \int_{2 B} \int_{2 B} \frac{(\lvert u(x)\rvert^2 + \lvert u (y) \rvert^2) \lvert \eta (x) - \eta (y) \rvert^2}{\lvert x - y \rvert^{d + 2 \alpha}} \; \d x \; \d y.
 \end{aligned}
\end{align}
For the time being, the fourth term on the right-hand side of~\eqref{Eq: Preparation for Young} is not estimated further. Concerning the fifth term on the right-hand side of~\eqref{Eq: Preparation for Young}, employ again Young's inequality to deduce
\begin{align}
\label{Eq: Absorbtion 2}
\begin{aligned}
 &\int_{2 B} \int_{2 B} \frac{\Re ([u (x) - u (y)] (\overline{u (y)} \eta(x) + \overline{u (x)} \eta (y))) (\eta (x) - \eta (y))}{\lvert x - y \rvert^{d + 2 \alpha}} \; \d x \; \d y \\
 &\leq \frac{1}{8 C_{\theta , \Lambda}} \int_{2 B} \int_{2 B} \frac{\lvert u (x) - u (y) \rvert^2 (\eta (x)^2 + \eta(y)^2)}{\lvert x - y \rvert^{d + 2 \alpha}} \; \d x \; \d y \\
 &\qquad + 2 C_{\theta , \Lambda} \int_{2 B} \int_{2 B} \frac{(\lvert u (x) \rvert^2 + \lvert u (y) \rvert^2) (\eta (x) - \eta (y))^2}{\lvert x - y \rvert^{d + 2 \alpha}} \; \d x \; \d y.
 \end{aligned}
\end{align}
Now, absorb each of the first terms on the right-hand sides of~\eqref{Eq: Absorbtion 1} and~\eqref{Eq: Absorbtion 2} to the left-hand side of~\eqref{Eq: Preparation for Young}. \par
The following terms remain on the right-hand side:
\begin{align*}
 \int_{2 B} \int_{2 B} \frac{(\lvert u (x) \rvert^2 + \lvert u (y) \rvert^2) (\eta (x) - \eta (y))^2}{\lvert x - y \rvert^{d + 2 \alpha}} \; \d x \; \d y \quad \text{and} \quad \int_{2 B} \lvert u (x) \rvert \; \d x \int_{\IR^d \setminus 2 B} \frac{\lvert u (y) \rvert}{\lvert x_0 - y \rvert^{d + 2 \alpha}} \; \d y.
\end{align*}
The second term is already one of the desired terms so we analyze the first term. Notice that by symmetry and the condition $\| \nabla \eta \|_{\L^{\infty}} \leq C_d / r$ we find
\begin{align*}
 \int_{2 B} \int_{2 B} \frac{(\lvert u (x) \rvert^2 + \lvert u (y) \rvert^2) (\eta (x) - \eta (y))^2}{\lvert x - y \rvert^{d + 2 \alpha}} \; \d x \; \d y &= 2\int_{2 B} \int_{2 B} \frac{\lvert u (x) \rvert^2 (\eta (x) - \eta (y))^2}{\lvert x - y \rvert^{d + 2 \alpha}} \; \d x \; \d y \\
 &\leq \frac{2 C_d^2}{r^2} \int_{2 B} \lvert u (x) \rvert^2 \int_{2 B} \lvert x - y \rvert^{2 (1 - \alpha) - d} \; \d y \; \d x \\
 &\leq \frac{C_{d , \alpha}}{r^{2 \alpha}} \int_{2 B} \lvert u (x) \rvert^2 \; \d x.
\end{align*}
Here, $C_{d , \alpha} > 0$ is a constant that depends only on $d$ and $\alpha$. \par
Summarizing everything, there exists a constant $C > 0$ depending only on $d$, $\alpha$, $\Lambda$, and $\theta$ such that
\begin{align*}
 \lvert \lambda \rvert \int_{\IR^d} \lvert u (x) \rvert^2 \eta (x)^2 \; \d x &+ \int_{2 B} \int_{2 B} \frac{\lvert u (x) - u (y) \rvert^2 (\eta (x)^2 + \eta (y)^2)}{\lvert x - y \rvert^{d + 2 \alpha}} \; \d x \; \d y \\
 &+ \int_{2 B} \int_{2 B} \frac{\lvert u (x) \eta (x) - u (y) \eta (y) \rvert^2}{\lvert x - y \rvert^{d + 2 \alpha}} \; \d x \; \d y \\
 &+ \int_{2 B} \int_{\IR^d \setminus 2 B} \frac{\lvert u (y) \rvert^2 \eta(y)^2}{\lvert x - y \rvert^{d + 2 \alpha}} \; \d x \; \d y + \int_{\IR^d \setminus 2 B} \int_{2 B} \frac{\lvert u (x) \rvert^2 \eta (x)^2}{\lvert x - y \rvert^{d + 2 \alpha}} \; \d x \; \d y \\
 &\qquad \leq C \bigg( \frac{1}{r^{2 \alpha}} \int_{2 B} \lvert u (x) \rvert^2 \; \d x + \int_{2 B} \lvert u (x) \rvert \; \d x \int_{\IR^d \setminus 2 B} \frac{\lvert u (y) \rvert}{\lvert x_0 - y \rvert^{d + 2 \alpha}} \; \d y \bigg).
\end{align*}
This proves the proposition.
\end{proof}

The following lemma brings the right-hand side of Caccioppoli inequality in Proposition~\ref{Prop: Caccioppoli} into a suitable form for the non-local weak reverse H\"older estimate.

\begin{lemma}
\label{Lem: Nonlocal}
Let $x_0 \in \IR^d$, $r > 0$, $\alpha \in (0 , 1)$, and $u \in \L^2_{\loc} (\IR^d)$. Then there exists a constant $C > 0$ depending only on $d$ and $\alpha$ such that
\begin{align*}
 \int_{2 B} \lvert u (x) \rvert \; \d x \int_{\IR^d \setminus 2 B} \frac{\lvert u (y) \rvert}{\lvert x_0 - y \rvert^{d + 2 \alpha}} \; \d y \leq C r^{d - 2 \alpha} \sum_{k = 1}^{\infty} 2^{- 2 \alpha k} \frac{1}{\lvert 2^{k + 1} B \rvert} \int_{2^{k + 1} B} \lvert u (y) \rvert^2 \; \d y.
\end{align*}
\end{lemma}

\begin{proof}
Jensen's inequality applied to the first integral followed by Young's inequality ensure for some constant $C > 0$ depending only on $d$ that
\begin{align*}
 \int_{2 B} \lvert u (x) \rvert \; \d x &\int_{\IR^d \setminus 2 B} \frac{\lvert u (y) \rvert}{\lvert x_0 - y \rvert^{d + 2 \alpha}} \; \d y \\
  &\leq C \bigg\{ \frac{1}{r^{2 \alpha}} \int_{2 B} \lvert u (x) \rvert^2 \; \d x + \bigg(r^{\frac{d}{2} + \alpha} \int_{\IR^d \setminus 2 B} \frac{\lvert u (y) \rvert}{\lvert x_0 - y \rvert^{d + 2 \alpha}} \; \d y \bigg)^2 \bigg\}.
\end{align*}
Furthermore, decomposing the second integral on the right-hand side into dyadic annuli yields
\begin{align*}
 \int_{\IR^d \setminus 2 B} \frac{\lvert u (y) \rvert}{\lvert x_0 - y \rvert^{d + 2 \alpha}} \; \d y = \sum_{k = 1}^{\infty} \int_{2^{k + 1} B \setminus 2^k B} \frac{\lvert u (y) \rvert}{\lvert x_0 - y \rvert^{d + 2 \alpha}} \; \d y.
\end{align*}
Now, apply Jensen's inequality to each of the integrals and further use that
\begin{align*}
 2^{k - 1} r \leq (2^k - 1) r \leq \lvert x_0 - y \rvert \qquad (y \in 2^{k + 1} B \setminus 2^k B)
\end{align*}
to establish for some constant $C > 0$ depending only on $d$ and $\alpha$
\begin{align*}
 \int_{\IR^d \setminus 2 B} \frac{\lvert u (y) \rvert}{\lvert x_0 - y \rvert^{d + 2 \alpha}} \; \d y \leq C \sum_{k = 1}^{\infty} 2^{- 2 \alpha k} r^{- 2 \alpha} \bigg( \frac{1}{\lvert 2^{k + 1} B \rvert} \int_{2^{k + 1} B} \lvert u (y) \rvert^2 \; \d y \bigg)^{\frac{1}{2}}.
\end{align*}
Finally, H\"older's inequality for series together with $\alpha > 0$ yield
\begin{align*}
 \bigg(r^{\frac{d}{2} + \alpha} \int_{\IR^d \setminus 2 B} \frac{\lvert u (y) \rvert}{\lvert x_0 - y \rvert^{d + 2 \alpha}} \; \d y \bigg)^2 \leq C r^{d - 2 \alpha} \sum_{k = 1}^{\infty} 2^{- 2 \alpha k} \frac{1}{\lvert 2^{k + 1} B \rvert} \int_{2^{k + 1} B} \lvert u (y) \rvert^2 \; \d y.
\end{align*}
This readily yields the desired estimate.
\end{proof}

\section{Proofs of Theorems~\ref{Thm: Resolvent} and~\ref{Thm: Maximal regularity}}
\label{Sec: Proofs}

\noindent Let $A$ denote the operator associated to the sesquilinear form
\begin{align*}
  \fa : \W^{\alpha , 2} (\IR^d) \times \W^{\alpha , 2} (\IR^d) &\to \IC \\
 (u , v) &\mapsto \int_{\IR^d} \int_{\IR^d} K (x , y) (u (x) - u (v)) \overline{(v(x) - v(y))} \; \d x \; \d y.
\end{align*}
Define $\Phi := \pi - \arccos (\Lambda^2)$. Notice that~\eqref{Eq: Ellipticity} implies that
\begin{align*}
 \fa (u , u) \in \S_{\pi - \Phi} \cup \{ 0 \} \qquad (u \in \W^{\alpha , 2} (\IR^d)).
\end{align*}
Thus, if $0 < \theta < \Phi$ and if $\lambda \in \S_{\theta}$, then an elementary trignometric consideration short that the sesquilinear form
\begin{align*}
  \fa_{\lambda} : \W^{\alpha , 2} (\IR^d) \times \W^{\alpha , 2} (\IR^d) \to \IC, \quad (u , v) \mapsto \lambda \int_{\IR^d} u \overline{v} \; \d x + \fa(u , v)
\end{align*}
is bounded and coercive. Thus, by the Lax--Milgram lemma, we find that $\lambda \in \rho(- A)$, the resolvent set of $- A$. Thus, for $f \in \L^2 (\IR^d)$ there exists a unique $u \in \dom(A)$ with $\lambda u + A u = f$. Testing this equation by $u$ yields by the same trigonometry consideration as above for a constant $C > 0$ depending only on $\theta$ and $\lambda$ that
\begin{align*}
 \lvert \lambda \rvert \| u \|_{\L^2 (\IR^d)}^2 + \int_{\IR^d} \int_{\IR^d} \frac{\lvert u (x) - u (y) \rvert^2}{\lvert x - y \rvert^{d + 2 \alpha}} \; \d x \; \d y \leq C \| f \|_{\L^2 (\IR^d)} \| u \|_{\L^2 (\IR^d)}.
\end{align*}
Now, forgetting about the double integral on the left-hand side and dividing by $\| u \|_{\L^2 (\IR^d)}$ shows that the $\L^2$ resolvent estimate
\begin{align}
\label{Eq: L2 Resolvent estimate}
 \| \lambda (\lambda + A)^{-1} f \|_{\L^2 (\IR^d)} = \lvert \lambda \rvert \| u \|_{\L^2 (\IR^d)} \leq C \| f \|_{\L^2 (\IR^d)}
\end{align}
is valid. \par
We remark, that for an operator $A$, the property of having maximal $\L^r$-regularity on the time interval $(0 , \infty)$ is in general \textit{stronger} than the property that $- A$ generates a bounded analytic semigroup. Indeed, combining the characterization of maximal $\L^r$-regularity via the notion of $\cR$-boundedness~\cite[Thm.~4.2]{Weis} and using the reformulation of $\cR$-boundedness via square function estimates if operators on $\L^p$-spaces are considered~\cite[Rem.~2.9]{Kunstmann_Weis} we arrive at the following statement. Namely, for any given $1 < r < \infty$ the operator $A$ has maximal $\L^r$-regularity if there exists $\theta > \pi / 2$ and a constant $C > 0$ such that for all $n_0 \in \IN$, $(\lambda_n)_{n = 1}^{n_0} \subset \S_{\theta}$, and $(f_n)_{n = 1}^{\infty} \subset \L^p (\IR^d)$ one has
\begin{align}
\label{Eq: Square function estimate}
 \Big\| \Big[ \sum_{n = 1}^{n_0} \lvert \lambda_n (\lambda_n + A)^{-1} f_n \rvert^2 \Big]^{1 / 2} \Big\|_{\L^p (\IR^d)} \leq C \Big\| \Big[ \sum_{n = 1}^{n_0} \lvert f_n \rvert^2 \Big]^{1 / 2} \Big\|_{\L^p (\IR^d)}.
\end{align}
Notice that $C$ has to be \textit{uniform} with respect to $n_0$, $(\lambda_n)_{n = 1}^{n_0} \subset \S_{\theta}$, and $(f_n)_{n = 1}^{\infty} \subset \L^p (\IR^d)$. Regarding the square root over the sum of squares as an Euclidean norm, this statement is equivalent to the boundedness of the following family of operator $\cT_{\theta}$ in the space $\Lop(\L^p (\IR^d ; \ell^2))$
\begin{align*}
 \cT_{\theta} := \{ (\lambda_1 (\lambda_1 + A)^{-1} , \dots , \lambda_{n_0} (\lambda_{n_0} + A)^{-1} , 0 , \dots) : n_0 \in \IN , (\lambda_n)_{n = 1}^{n_0} \subset \S_{\theta} \}.
\end{align*}
Here, an operator $T \in \cT_{\theta}$ acts on a function $f = (f_n)_{n \in \IN} \in \L^p (\IR^d ; \ell^2)$ via
\begin{align*}
 T f = (\lambda_1 (\lambda_1 + A)^{-1} f_1 , \dots , \lambda_{n_0} (\lambda_{n_0} + A)^{-1} f_{n_0} , 0 , \dots).
\end{align*}

Notice that the square function estimate~\eqref{Eq: Square function estimate} is in the case $p = 2$ equivalent to the uniform resolvent estimate in~\eqref{Eq: L2 Resolvent estimate}. Thus, we already know that the family of operators $\cT_{\theta}$ is bounded in $\Lop(\L^2 (\IR^d ; \ell^2))$. Let $\theta \in (0 , \Phi)$. We show in the following, that there exists $\eps > 0$ such that for all $p \geq 2$ that satisfy
\begin{align}
\label{Eq: Condition p}
 \Big\lvert \frac{1}{p} - \frac{1}{2} \Big\rvert < \frac{\alpha}{d} + \eps
\end{align}
the family $\cT_{\theta}$ is bounded in $\Lop(\L^p (\IR^d ; \ell^2))$. To this end, we verify that each operator in $\cT_{\theta}$ fulfills the assumptions of Theorem~\ref{Thm: Shens Lp extrapolation theorem, Banach space valued} with uniform constants for all operators in $\cT_{\theta}$. As the knowledge of the $\L^q$-operator norm in Theorem~\ref{Thm: Shens Lp extrapolation theorem, Banach space valued} is known to depend only on the quantities at stake, this will imply that $\cT_{\theta}$ is bounded in $\Lop(\L^p (\IR^d ; \ell^2))$. As the $\L^2(\IR^d ; \ell^2)$-boundedness is already established we concentrate on the non-local weak reverse H\"older estimate, i.e., estimate~\eqref{Eq: Weak reverse Hoelder inequality}. \par
To this end, let $x_0 \in \IR^d$ and $r > 0$. Let further $n_0 \in \IN$, $\lambda_1 , \dots , \lambda_{n_0} \in \S_{\theta}$, and let $f_1 , \dots , f_{n_0} \in \L^{\infty} (\IR^d)$ have compact support and be such that $f_n \equiv 0$ in $B(x_0 , 2 r)$ for all $n = 1 , \dots , n_0$. Define
\begin{align*}
 u_n := (\lambda_n + A)^{-1} f_n \qquad (n = 1 , \dots , n_0).
\end{align*}
Let $2 < p < \infty$ satisfy
\begin{align*}
 0 < \frac{1}{2} - \frac{1}{p} \leq \frac{\alpha}{d}.
\end{align*}
Choose $0 < \vartheta \leq \alpha$ such that
\begin{align*}
 \frac{1}{2} - \frac{1}{p} = \frac{\vartheta}{d}, \quad \text{i.e.,} \quad \W^{\vartheta , 2} (\IR^d) \hookrightarrow \L^p (\IR^d).
\end{align*}
Choose $\beta \in (0 , 1]$ that satisfies $\alpha \beta = \vartheta$. With this choice, the interpolation inequality
\begin{align*}
 \| v \|_{\L^p (\IR^d)} \leq C \| v \|_{\L^2 (\IR^d)}^{1 - \beta} \| v \|_{\W^{\alpha , 2}}^{\beta} \qquad (v \in \W^{\alpha , 2} (\IR^d))
\end{align*}
holds. Notice that by scaling, even the homogeneous counterpart of this interpolation inequality is valid, namely,
\begin{align*}
 \| v \|_{\L^p (\IR^d)} \leq C \bigg( \int_{\IR^d} \lvert v \rvert^2 \; \d x \bigg)^{\frac{1 - \beta}{2}} \bigg( \int_{\IR^d} \int_{\IR^d} \frac{\lvert v (x) - v (y) \rvert^2}{\lvert x - y \rvert^{d + 2 \alpha}} \; \d x \; \d y \bigg)^{\frac{\beta}{2}} \qquad (v \in \W^{\alpha , 2} (\IR^d)).
\end{align*}
Let $\eta \in \C_c^{\infty} (B(x_0 , 3 r / 2))$ with $0 \leq \eta \leq 1$, $\eta \equiv 1$ in $B(x_0 , r)$, and $\| \nabla \eta \|_{\L^{\infty}} \leq C_d / r$ for some constant $C_d > 0$ depending only on $d$. Define
\begin{align*}
 v := \bigg[ \sum_{n = 1}^{n_0} \lvert \lambda_n \rvert^2 \lvert u_n \eta \rvert^2 \bigg]^{\frac{1}{2}}.
\end{align*}
Applying the interpolation inequality above to $v$ then yields together with the properties of $\eta$
\begin{align*}
 &\bigg( \int_{B(x_0 , r)} \bigg[ \sum_{n = 1}^{n_0} \lvert \lambda_n \rvert^2 \lvert u_n (x) \rvert^2 \bigg]^{\frac{p}{2}} \; \d x \bigg)^{\frac{1}{p}} \\
 &\quad \leq \bigg( \int_{\IR^d} \bigg[ \sum_{n = 1}^{n_0} \lvert \lambda_n \rvert^2 \lvert u_n (x) \eta (x) \rvert^2 \bigg]^{\frac{p}{2}} \; \d x \bigg)^{\frac{1}{p}} \\
 &\quad \leq C \bigg( \int_{\IR^d} \sum_{n = 1}^{n_0} \lvert \lambda_n \rvert^2 \lvert u_n (x) \eta (x) \rvert^2 \; \d x \bigg)^{\frac{1 - \beta}{2}} \bigg( \sum_{n = 1}^{n_0} \lvert \lambda_n  \rvert^2 \int_{\IR^d} \int_{\IR^d} \frac{\lvert u_n (x) \eta (x) - u_n (y) \eta (y) \rvert^2}{\lvert x - y \rvert^{d + 2 \alpha}} \; \d x \; \d y \bigg)^{\frac{\beta}{2}}.
\end{align*}
Now, use that $\supp(\eta) \subset B(x_0 , 3 r / 2)$ and deduce by symmetry that
\begin{align*}
 \int_{\IR^d} \int_{\IR^d} \frac{\lvert u_n (x) \eta (x) - u_n (y) \eta (y) \rvert^2}{\lvert x - y \rvert^{d + 2 \alpha}} \; \d x \; \d y &= \int_{2 B} \int_{2 B} \frac{\lvert u_n (x) \eta (x) - u_n (y) \eta (y) \rvert^2}{\lvert x - y \rvert^{d + 2 \alpha}} \; \d x \; \d y \\
 &\qquad + 2 \int_{\IR^d \setminus 2 B} \int_{2 B} \frac{\lvert u_n (x) \rvert^2  \eta (x)^2}{\lvert x - y \rvert^{d + 2 \alpha}} \; \d x \; \d y.
\end{align*}
Apply Proposition~\ref{Prop: Caccioppoli} together with Lemma~\ref{Lem: Nonlocal} to each of these summands and finally deduce
\begin{align*}
 &\bigg( \int_{B(x_0 , r)} \bigg[ \sum_{n = 1}^{n_0} \lvert \lambda_n \rvert^2 \lvert u_n (x) \rvert^2 \bigg]^{\frac{p}{2}} \; \d x \bigg)^{\frac{1}{p}} \leq C r^{\frac{d}{2} - \alpha \beta} \bigg( \sum_{k = 1}^{\infty} 2^{- 2 \alpha k} \fint_{B(x_0 , 2^{k + 1} r)} \sum_{n = 1}^{n_0} \lvert \lambda_n \rvert^2 \lvert u_n (x) \rvert^2 \; \d x \bigg)^{\frac{1}{2}}.
\end{align*}
Now, since
\begin{align*}
 \frac{d}{2} - \alpha \beta = d \Big(\frac{1}{2} - \frac{\vartheta}{d}\Big) = \frac{d}{p}
\end{align*}
one can divide the previous estimate by $r^{d / p}$ and obtain the desired non-local weak reverse H\"older estimate. The non-local Gehring lemma proven in~\cite[Thm.~2.2]{Auscher_Bortz_Egert_Saari_2} now implies the existence of $\eps > 0$ depending only on $d$, $\alpha$, $\theta$, $\Lambda$, and $p$ such that
\begin{align*}
 &\bigg( \fint_{B(x_0 , r)} \bigg[ \sum_{n = 1}^{n_0} \lvert \lambda_n \rvert^2 \lvert u_n (x) \rvert^2 \bigg]^{\frac{p + \eps}{2}} \; \d x \bigg)^{\frac{1}{p + \eps}} \leq C \bigg( \sum_{k = 1}^{\infty} 2^{- 2 \alpha k} \fint_{B(x_0 , 2^{k + 2} r)} \sum_{n = 1}^{n_0} \lvert \lambda_n \rvert^2 \lvert u_n (x) \rvert^2 \; \d x \bigg)^{\frac{1}{2}}
\end{align*}
holds. Clearly, the right-hand side can be estimated by
\begin{align*}
 \bigg( \sum_{k = 1}^{\infty} 2^{- 2 \alpha k} \fint_{B(x_0 , 2^{k + 2} r)} \sum_{n = 1}^{n_0} \lvert \lambda_n \rvert^2 \lvert u_n (x) \rvert^2 \; \d x \bigg)^{\frac{1}{2}} \leq C \sup_{B^{\prime} \supset B(x_0 , r)} \bigg( \fint_{B^{\prime}} \sum_{n = 1}^{n_0} \lvert \lambda_n \rvert^2 \lvert u_n (x) \rvert^2 \; \d x \bigg)^{\frac{1}{2}}.
\end{align*}
By $B^{\prime}$ we denote here an arbitrary ball in $\IR^d$ containing $B(x_0 , r)$. \par
This implies, that each operator $T \in \cT_{\theta}$ fulfills a non-local weak reverse H\"older estimate with uniform constants. We conclude the statements of Theorems~\ref{Thm: Resolvent} and~\ref{Thm: Maximal regularity} in the case $p \geq 2$. Remark, that the conclusion of Theorem~\ref{Thm: Resolvent} follows from above by taking $n_0 = 1$. \par
To conclude the statements of the theorems for $p \leq 2$ satisfying~\eqref{Eq: Condition p} we argue by duality. Notice that the adjoint operator of $A$ belongs to the same class of operators as it is associated with the sesquilinear form
\begin{align*}
 \fb : \W^{\alpha , 2} (\IR^d) \times \W^{\alpha , 2} (\IR^d) \to \IC, \quad (u , v) \mapsto \int_{\IR^d} \int_{\IR^d} \overline{K (y , x)} (u (x) - u(y)) \overline{(v (x) - v (y))} \; \d x \; \d y.
\end{align*}
In particular, $\overline{K (y , x)}$ fulfills the ellipticity assumption~\eqref{Eq: Ellipticity} for the same constant $\Lambda$ as $K$ did. Now, since the dual space of $\L^p(\IR^d ; \ell^2)$ is $\L^{p^{\prime}} (\IR^d ; \ell^2)$ and $p^{\prime} > 2$ satisfies~\eqref{Eq: Condition p} we conclude the statements of Theorems~\ref{Thm: Resolvent} and~\ref{Thm: Maximal regularity} in the case $p < 2$ as well. \qed


\end{document}